\documentclass[12pt]{amsart}
\usepackage{amssymb,amsmath,url,graphicx, mathtools}
\usepackage[usenames,dvipsnames,svgnames,table]{xcolor}
\usepackage{bm}

\usepackage[all]{xy}
\usepackage{enumitem}
\usepackage{graphicx}
\usepackage[hidelinks]{hyperref}
\usepackage[margin=1in]{geometry}
\usepackage{bbm}
\usepackage{aliascnt}
\usepackage{cleveref}

\newif\ifinfootnote
\infootnotefalse

\let\footnoteasusual\footnote
\renewcommand{\footnote}[1]
{\infootnotetrue\footnoteasusual{#1}\infootnotefalse}

\newif\ifdebug
\debugfalse

\newcommand{\onote}[1] {\ifdebug {{\color{blue} \footnotesize #1}} \else \fi}
\newcommand{\N}{\mathbb{N}}
\newcommand{\Z}{\mathbb{Z}}
\newcommand{\Q}{\mathbb{Q}}
\newcommand{\R}{\mathbb{R}}
\newcommand{\C}{\mathbb{C}}

\newcommand{\bigslant}[2]{{\raisebox{.2em}{$#1$}\left/\raisebox{-.2em}{$#2$}\right.}}

\DeclareMathOperator{\conv}{conv}
\DeclareMathOperator{\vol}{vol}
\def \GL {\mathsf{GL}}

\def \supp {\operatorname{supp}}

\def \calW {{\mathcal W}}
\def \calO {{\mathcal O}}

\def \tQ {\tilde{Q}}

\setlist{topsep=0pt,itemsep=6pt}

\swapnumbers
\numberwithin{equation}{section}

\crefname{theorem}{theorem}{theorems}
\Crefname{theorem}{Theorem}{Theorems}
\crefname{corollary}{corollary}{corollaries}
\Crefname{corollary}{Corollary}{Corollaries}
\crefname{lemma}{lemma}{lemmas}
\Crefname{lemma}{Lemma}{Lemmas}
\crefname{proposition}{proposition}{propositions}
\Crefname{proposition}{Proposition}{Propositions}
\crefname{definition}{definition}{definitions}
\Crefname{definition}{Definition}{Definitions}
\crefname{claim}{claim}{claims}
\Crefname{claim}{Claim}{Claims}
\crefname{conjecture}{conjecture}{conjectures}
\Crefname{conjecture}{Conjecture}{Conjectures}
\crefname{construction}{construction}{constructions}
\Crefname{construction}{Construction}{Constructions}
\crefname{fact}{fact}{facts}
\Crefname{fact}{Fact}{Facts}
\crefname{notation}{notation}{notations}
\Crefname{notation}{Notation}{Notations}
\crefname{remark}{remark}{remarks}
\Crefname{remark}{Remark}{Remarks}
\crefname{example}{example}{examples}
\Crefname{example}{Example}{Examples}
\crefname{convention}{convention}{conventions}
\Crefname{convention}{Convention}{Conventions}

\newtheorem {theorem}{Theorem}[section]
\newaliascnt{thm}{theorem}

\aliascntresetthe{thm}

\newaliascnt{lemma}{theorem}
\newtheorem {lemma}[lemma]{Lemma}
\aliascntresetthe{lemma}

\newaliascnt{claim}{theorem}

\aliascntresetthe{claim}

\newtheorem {claim*}{Claim}

\newaliascnt{conjecture}{theorem}

\aliascntresetthe{conjecture}

\newaliascnt{corollary}{theorem}
\newtheorem {corollary}[corollary]{Corollary}
\aliascntresetthe{corollary}

\newaliascnt{proposition}{theorem}
\newtheorem {proposition}[proposition]{Proposition}
\aliascntresetthe{proposition}

\newtheorem*{proposition*}{Proposition}

\newaliascnt{construction}{theorem}

\aliascntresetthe{construction}

\newaliascnt{fact}{theorem}

\aliascntresetthe{fact}

\theoremstyle{definition}
\newaliascnt{definition}{theorem}
\newtheorem{definition}[definition]{Definition}
\aliascntresetthe{definition}

\newaliascnt{notation}{theorem}
\newtheorem{notation}[notation]{Notation}
\aliascntresetthe{notation}

\theoremstyle{remark}
\newaliascnt{remark}{theorem}
\newtheorem{remark}[remark]{Remark}
\aliascntresetthe{remark}

\newtheorem*{remark*}{Remark}
\newaliascnt{example}{theorem}

\aliascntresetthe{example}

\newaliascnt{convention}{theorem}

\aliascntresetthe{convention}

\makeatletter
\newcommand{\tpitchfork}{%
  \vbox{
    \baselineskip\z@skip
    \lineskip-.52ex
    \lineskiplimit\maxdimen
    \m@th
    \ialign{##\crcr\hidewidth\smash{$-$}\hidewidth\crcr$\pitchfork$\crcr}
  }%
}
\makeatother

\usepackage[backend=biber]{biblatex}
\newcommand{\AZGL}[1]{\mathrm{A}_{\mathbb Z}\mathrm{GL}(#1,\mathbb Z)}

\def \on {\operatorname}
\def \Int {\on{Int}}
\def \RInt {\on{relint}}

\def \Vol {\on{vol}}

\begin{document}

\title{Ehrhart theorem for integral-integral affine manifolds with corners}
\author{Oded Elisha}
\address{School of Mathematical Sciences, Tel Aviv University}
\email{odedelisha@mail.tau.ac.il}
\date{\today}


\begin{abstract}
 We study lattice-counting functions on compact integral-integral affine manifolds with corners, which generalize integral unimodular polytopes. We prove analogues of Ehrhart’s theorem and Ehrhart--Macdonald reciprocity in this setting. 
\end{abstract}

\maketitle

\section{Introduction}

In Ehrhart theory, one studies the lattice-counting function $L_P(m)=\#(m^{-1}\Z^n \cap P)$ where $P\subseteq \R^n$ is a rational polytope. In this paper we consider compact integral-integral affine manifolds with corners. These spaces generalize integral unimodular polytopes while still allowing to define a lattice-counting function.
An integral-integral affine manifold with corners $M$ is modeled on the standard orthant $\R_{\ge 0}^n$; where transition maps are locally of the form 
\begin{equation} \label{eq:iia}
x \mapsto Ax +b
\quad 
\text{ with $A \in \GL(n,\Z)$ and $b \in \Z^n$ }\,. 
\end{equation}
An \textbf{integral chart} is a chart of the integral-integral affine structure with corners, the set of lattice points of order $m\in \N$ in $M$ is
    \begin{multline*}
        M_{{m^{-1}}\Z} \coloneq \{x\in M \mid \text{ there exists 
         an integral chart } \psi \colon U\rightarrow \Omega \subseteq \R_{\ge 0}^n \\
         \text{ whose domain contains $x$ and } \psi(x)\in m^{-1}\Z^n\}
    \end{multline*}
and the \textbf{lattice-counting function} is
\begin{align*}
    &L_M \colon \N \rightarrow \N \cup \{0,\infty\} \\
    &L_M(m)\coloneq\#M_{m^{-1}\Z}\,.
\end{align*}

Ehrhart's theorem states that, if $P\subseteq \R^n$ is an integral polytope, then its lattice counting function $L_P(m)$ is polynomial. Furthermore, the Ehrhart-Macdonald reciprocity is \[
    L_{\RInt P}(m)=(-1)^{\dim P}L_P(-m)\,,
\]
where $\RInt P$ is the interior of $P$ relative to its affine span.

We extend both theorems for compact integral-integral affine manifolds with corners. We prove that the lattice-counting function $L_M$ is a polynomial and that evaluating $L_M$ at the negative integers is, up to sign, the lattice-counting function of the interior (\Cref{generalized_ehrhart}). As in the case of integral polytopes, we show that the leading coefficient is the total (affine) volume and the constant coefficient is the Euler characteristic (\Cref{leading_constant_coefficients}).

The general idea is as follows: First, show that a compact integral-integral affine manifold with corners can be covered by a finite collection of rational polytopes, whose non-empty intersections are also rational polytopes. By Ehrhart theorem for rational polytopes and inclusion-exclusion, this gives that the lattice-counting function of the manifold is quasi-polynomial.
Then, use some form of asymptotics in order to conclude further properties of the lattice-counting function.

This idea, but for closed integral-integral affine manifolds, was used in \cite{elisha2026integralpointsvolumeintegralintegral} along with the Poisson summation formula. In this paper, we instead use local Euler-Maclaurin-type asymptotics that are due to Berline and Vergne \cite{berline_local_2016}.

One motivation for this project comes from symplectic geometry, where a special case of a generalization of Ehrhart's theorem was used in the author's previous paper \cite[Corollary 3.14]{elisha2026integralpointsvolumeintegralintegral} to prove a theorem about closed integral-integral affine manifolds. There, the integral-integral affine structure arises naturally on the base of a regular Lagrangian fibration equipped with a prequantization line bundle (see Karshon, Hamilton and Yoshida \cite[Section 4]{hamilton_integral-integral_2024}). A more general Ehrhart theory for compact integral-integral affine 
manifolds with corners may be useful in studying prequantized Lagrangian fibrations
with elliptic singularities (see Sepe and Vũ Ngoc \cite[Remark 4.20]{sepe_integrable_2018}, Maarten Mol \cite{mol2023} and Maarten Mol and Rui Fernandes \cite{fernandes2026}).
Related asymptotic methods in the context of geometric quantization appear in Loizides, Paradan and Vergne \cite{loizides_semi-classical_2021}.

A second motivation comes from gluing the facets of integral polytopes. More generally, if $\Gamma$ is a group of transformations of the form (\ref{eq:iia}) acting freely and properly (on $\R^n$), then the quotient $\bigslant{\R^n}{\Gamma}$ is an integral-integral affine manifold. In many examples, a fundamental domain for this quotient is an integral polytope. In this case, some integral polytopes contained in the fundamental domain project to compact integral-integral affine manifolds with corners. For example, when $\Gamma=\{x\mapsto x+a\mid a\in 3\Z^2\}$, the quotient is a torus and the rectangle $[0,3]\times [1,2]$ projects to a cylinder $C$. This cylinder is obtained from an integral unimodular polytope with two of its facets glued and its lattice-counting function is well-defined and polynomial. In particular, its lattice-counting function is
\[
    L_{C}(m)=3m^2+3m=\Vol(C)m^2+3m+\chi(C)
\]
and its interior's lattice-counting function is
\[
    L_{\Int C}(m)=3m^2-3m=\Vol(\Int C)m^2-3m+\chi(\Int C)\,.
\]

Further examples arise by taking a compact integral--integral affine manifold with corners $F$ and an integral--integral affine automorphism $\psi:F\rightarrow F$. The space \[
    M\coloneq \bigslant{(\R \times F)}{\sim},\quad (t,x)\sim (t+1, \psi(x))
\] is a compact integral--integral affine manifold with corners and its lattice-counting function is $L_M(m)=L_{S^1}(m)\cdot L_{F}(m)=mL_F(m)$.

When $F=S^1\times [0,1]$ and $\psi(\varphi,x)=(\varphi+x,x)$ we obtain a manifold with boundary diffeomorphic to $\mathbb{T}^2\times [0,1]$. The lattice-counting function is $L_M(m)=m^2(m+1)$ and $M$ has a non-trivial linear monodromy group which is neither finite nor orthogonal.
When $F=[0,1]^2$ and $\psi(x,y)=(1-x,y)$ we obtain a manifold with corners diffeomorphic to the product of a Möbius strip with $[0,1]$. Moreover, its linear monodromy group is finite and does not preserve orientation, and its lattice-counting function is $L_M(m)=m(m+1)^2$.

\subsection*{Acknowledgements}

I thank Yael Karshon for research guidance and for insightful discussions. 
I thank Rui Fernandes, Daniele Sepe and Camilo Arias Abad for suggesting that Ehrhart's theorem can be generalized to integral--integral affine manifolds with corners by rescaling the integral--integral affine structure.
Another special thanks goes to Yiannis Loizides, who introduced us to the use of Ehrhart theory with asymptotics.

O.\ Elisha is partly supported by Yaron Ostrover's
ISF Grant 938/22.

\section{Rational polyhedra}

\begin{notation}
    Our convention for the set of natural numbers is 
    \[\N \coloneq \{ 1, 2, \ldots \}\,.\] 
    For every $n\in \N$, \[[n]\coloneq \{1,\dots, n\}\,.\]
\end{notation}
\begin{notation}
    \textbf{The standard orthant} in $\R^n$ is \[
        \R_{\ge 0}^n \coloneq \{x\in \R^n \mid x_i \ge 0 \;\forall 1 \le i \le n\}
    \]
    and its interior is
    \[
        \R_{>0}^n\coloneq \{x\in \R^n \mid x_i>0\; \forall 1 \le i \le n\}\,.
    \]
\end{notation}

\begin{notation}
$\AZGL{n}$ is 
the set of maps $\R^n \to \R^n$ of the form 
$x \mapsto Ax +b$ with $A \in \GL(n,\Z)$ and $b \in \Z^n$.
\textbf{An integral-integral affine map} is an element $\psi\in \AZGL{n}$.
\end{notation}

\begin{definition}
A map $f \colon \Omega \rightarrow \R^n$ on an open subset $\Omega \subseteq \R^n$ is called \textbf{locally integral-integral affine} if 
each point has a neighbourhood on which the map
agrees with an integral-integral affine map.
\end{definition}

\begin{definition}
\label{polytope_def}
    A set $P\subseteq \R^n$ is a \textbf{rational polytope} if there exist $v_1,\dots,v_k \in \Q^n$, such that 
    $P=\conv\{v_1,\dots,v_k\}=\{\sum_{i=1}^k{\alpha_iv_i}\mid \alpha_i \ge 0,\; \sum_{i=1}^k{\alpha_i}=1\}$.

    The \textbf{dimension of $P$} is the dimension of the affine span of $P$, denoted $\dim P$.

    The \textbf{relative interior} of $P$, denoted $\RInt P$, is its interior relative to its affine span.

    A point $x\in P$ is called a \textbf{vertex} of $P$ if $x$ is an extreme point in P.
    That means that if $x=\lambda x_1+(1-\lambda) x_2$ for $0 \le \lambda \le 1$ where $x_1,x_2\in P$ and $x_1\neq x_2$, then $\lambda=0$ or $\lambda =1$.
\end{definition}

\begin{definition}
    \label{lattice_count_function}
    Let $P\subseteq \R^n$ be a rational polytope. Its \textbf{lattice-counting function} is
    \begin{align*} &L_P \colon \N \rightarrow \N \cup \{0\}\\
        &L_P(m)\coloneq\#((m^{-1}\Z^n) \cap P) = |\{x \in P \mid mx \in \Z^n \}| \,.
    \end{align*}
    Similarly the lattice-counting function of its relative interior is
    \begin{align*}
        &L_{\RInt P}(m)\coloneq\#((m^{-1}\Z^n) \cap \RInt P) = |\{x \in \RInt P \mid mx \in \Z^n \}| \,.
    \end{align*}
\end{definition}

\begin{definition}  \label{quasi-polynomial}
A function $f \colon \Z \rightarrow \R$ is \textbf{quasi-polynomial} if
there exist $n \in \N \cup\{0\}$
and periodic functions $\{a_k \colon \Z \rightarrow \R\}_{k=0}^n$ such that for all $m\in \Z$
    \[\displaystyle f(m)=\sum_{k=0}^n{a_k(m)m^k}\,.\]
A number $d\in \N$ is \textbf{a period of $f$} if $f(md+k)$ is a polynomial in $m$ for all $0\le k<d$.
\end{definition}

\begin{theorem}[Rational Ehrhart]
    \label{rational_ehrhart}
    Let $P\subseteq\R^n$ be a rational polytope. 
    Then
    \begin{enumerate}
        \item $L_P(m)$ coincides with a quasi-polynomial with $a_k\equiv 0$ for all $k>\dim P$.
        \item The least common multiple of all the denominators of all the coordinates of the vertices of $P$ is a period of $L_P$.
        \item If $P$ is $n$-dimensional then the leading term is $\Vol(P) \cdot m^n$.
        \item (Ehrhart-Macdonald reciprocity) $L_{\RInt P}(m)=(-1)^{\dim P}L_P(-m)$.
        \item If $P$ is a lattice polytope (i.e. every vertex lies in $\Z^n$) then $L_P$ coincides with a polynomial whose constant coefficient is 1.
    \end{enumerate}
\end{theorem}
\begin{proof}
    See Beck and Robins \cite[Propositions 3.8, 3.19, 3.20, 3.15, 3.23, 3.34, 4.1]{sinai_robins}.
\end{proof}

\begin{proposition}[Orthant asymptotics]
    \label{orthant_asymptotics}
    Let $h: \R^n \rightarrow \R$ be a smooth compactly supported function. Then there exist $a_0,\dots,a_n\in \R$ such that \[
        \sum_{x\in m^{-1}\Z^n \cap \R_{\ge 0}^n}{h(x)}=\sum_{k=0}^na_km^k+o(1)
    \]
    as $m \rightarrow \infty$.
\end{proposition}
\begin{proof}
    By Berline and Vergne \cite[Theorem 5.5]{berline_local_2016} the function $h$ along with the standard orthant and the standard integer lattice admit an asymptotic expansion \[
        \frac{1}{m^n}\sum_{x\in \R_{\ge 0}^n \cap \Z^n}{h(\frac{x}{m})}\sim \sum_{k=0}^\infty{\frac{b_k}{m^k}}
    \]
    as $m \rightarrow \infty$.
    In particular,
    \[
        \frac{1}{m^n}\sum_{x\in m^{-1}\Z^n\cap \R^n_{\ge 0}}{h(x)}=\sum_{k=0}^n{\frac{b_k}{m^k}}+O\left(\frac{1}{m^{n+1}}\right)
    \]
    as $m\rightarrow \infty$. Then multiplying both sides by $m^n$ 
    \[
        \sum_{x\in m^{-1}\Z^n\cap\R^n_{\ge 0}}{h(x)}=\sum_{k=0}^n{b_km^{n-k}}+O\left(\frac{1}{m}\right)
    \]
    as $m\rightarrow \infty$.
    The claim then follows by setting $a_k\coloneq b_{n-k}$.
\end{proof}

\begin{lemma}{(Local Reciprocity)}
    \label{local_reciprocity}
    Let $h:\R^n \rightarrow \R$ be a smooth compactly supported function and let $a_0,\dots,a_n\in \R$ be such that \begin{align*}
        \sum_{x\in m^{-1}\Z^n\cap \R_{\ge 0}^n}{h(x)}=\sum_{k=0}^n{a_km^k}+o(1)\,.
    \end{align*} Then \[
        \sum_{x\in m^{-1}\Z^n\cap \R^n_{>0}}{h(x)}=(-1)^n\sum_{k=0}^n{a_k(-m)^k}+o(1)\,.
    \]
\end{lemma}
\begin{proof}
    For each $I\subseteq [n]$ the associated face of the standard orthant is
    \[
        F_I\coloneq \{x\in \R^n_{\ge 0} \mid x_i=0 \;\forall i\in I\}\,.
    \]
    In particular, $F_{\emptyset}=\R_{\ge 0}^n$. Since the boundary of the standard orthant is the union of its facets (i.e., $\partial (\R_{\ge 0}^n)=\bigcup_{i=1}^n{F_{\{i\}}}$) we have by inclusion-exclusion
    \begin{align}
        \label{indicator_inclusion_exclusion}
        \mathbbm{1}_{\R_{>0}^n}=\mathbbm{1}_{\R_{\ge 0}^n}-\sum_{\emptyset \neq I\subseteq [n]}{(-1)^{|I|+1}\mathbbm{1}_{F_I}}\,.
    \end{align}
    For $C=F_I$ or $C=\R_{>0}^n$, the exponential generating function \begin{align*}
        \sum_{x\in \Z^n\cap C} e^{\langle \xi,x\rangle}
    \end{align*}
    converges for every $\xi \in \C^n$ such that $\Re(\xi_i)< 0$ for every $1\le i\le n$. This function 
    has an analytic continuation to a meromorphic function $S(C)$ on $\C^n$ (see Berline and Vergne \cite[Section 2.3]{berline_local_2016}).
    In fact, for every $I\subseteq [n]$
    \begin{align*}
        &\sum_{x\in \Z^n\cap F_I}e^{\langle \xi,x\rangle}=\prod_{i\in [n]\setminus I}\frac{1}{1-e^{\xi_i}}=S(F_I)(\xi)
    \end{align*}
    and
    \begin{align*}
        &\sum_{x\in \Z^n\cap \R^n_{> 0}}e^{\langle \xi,x\rangle}=\prod_{i=1}^n\frac{e^{\xi_i}}{1-e^{\xi _i}}=S(\R_{>0}^n)(\xi)\,.
    \end{align*}
    Therefore by (\ref{indicator_inclusion_exclusion}),
    \begin{align*}
        S(\R_{>0}^n)(\xi)=S(\R_{\ge 0}^n)(\xi)-\sum_{\emptyset \neq I\subseteq [n]}{(-1)^{|I|+1}S({F_I})(\xi)}\,.
    \end{align*}
    Because
    \begin{align}
        \label{stanley_reciprocity}
        (-1)^nS(\R_{\ge0}^n)(-\xi)=(-1)^n\prod_{i=1}^n\frac{1}{1-e^{-\xi_i}}=\prod_{i=1}^n\frac{e^{\xi_i}}{1-e^{\xi_i}}=S(\R_{>0}^n)(\xi)\,,
    \end{align}
    we obtain
    \begin{align}
        \label{stanley_reciprocity_orthant}
        (-1)^nS(\R_{\ge 0}^n)(-\xi)=S(\R_{\ge 0}^n)(\xi)-\sum_{\emptyset \neq I\subseteq [n]}{(-1)^{|I|+1}S({F_I})(\xi)}
    \end{align}
    as meromorphic functions on $\C^n$.

    Applying \Cref{orthant_asymptotics} to each face $F_I$ and by (\ref{indicator_inclusion_exclusion}) there exist $b_0,\dots,b_n\in \R$ such that
    \begin{multline*}
        \sum_{x\in m^{-1}\Z^n\cap \R_{>0}^n}{h(x)}=\sum_{x\in m^{-1}\Z^n\cap \R_{\ge 0}^n}{h(x)}-\sum_{\emptyset \neq I \subseteq [n]}{(-1)^{|I|+1}\sum_{x\in m^{-1}\Z^n\cap F_I}{h(x)}}=\sum_{k=0}^n{b_km^k}+o(1)\,.
    \end{multline*}

    Let $\xi\in \C^n$ be such that $\xi_i\neq 0$ for all $1 \le i \le n$. Then the function $z\mapsto S(C)(z\xi)$ is meromorphic so it has Laurent series expansion in a small punctured disk centered at the origin. For every $k\in \Z$ the $(-k)$-th coefficient $S(C)_{[-k]}(\xi)$ of this Laurent series is a rational homogeneous function in $\xi$ of degree $-k$ (see Berline and Vergne \cite[Section 2,3]{berline_local_2016}).
    
    By Berline and Vergne \cite[Theorem 3.5]{berline_local_2016} the limits
    \[
        T_k(C)\coloneq \lim_{\epsilon\rightarrow 0,\epsilon >0}{S(C)_{[-k]}(-i\xi+\epsilon\lambda)}\,,
    \]
    where $\lambda=(-1,\dots,-1)$, exist as tempered distributions for all $C=F_I$, and the coefficients in the above asymptotic expansions are given by
    \begin{align*}
        a_k&=\mathcal{F}^{-1}(T_k(\R_{\ge 0}^n))\;h \\b_k&=\mathcal{F}^{-1}(T_k(\R_{\ge 0}^n))\;h -\sum_{\emptyset \neq I \subseteq [n]}{(-1)^{|I|+1}\mathcal{F}^{-1}(T_k(F_I))\;h}\\
        &\qquad=\mathcal{F}^{-1}(\left(T_k(\R_{\ge 0}^n)-\sum_{\emptyset \neq I\subseteq [n]}{(-1)^{|I|+1}T_k({F_I})}\right))\;h
    \end{align*}
    where $\mathcal{F}$ is the Fourier transform of tempered distributions.
    
    By (\ref{stanley_reciprocity_orthant}),
    \begin{align*}
        \left(S(\R_{\ge 0}^n)(\xi)-\sum_{\emptyset \neq I\subseteq [n]}{(-1)^{|I|+1}S({F_I})(\xi)}\right)_{[-k]}&=(-1)^nS(\R^n_{\ge 0})(-\xi)_{[-k]}
        \\&=(-1)^{n-k}S(\R^n_{\ge 0})(\xi)_{[-k]}\,.
    \end{align*}
    Consequently,
    \begin{align*}
        b_k&=\mathcal{F}^{-1}\left(T_k(\R_{\ge 0}^n)-\sum_{\emptyset \neq I\subseteq [n]}{(-1)^{|I|+1}T_k({F_I})}\right)\;h\\
        &=\mathcal{F}^{-1}\left((-1)^{n-k}\,T_k(\R_{\ge 0}^n)\right)\;h\\&=(-1)^{n-k}a_k=(-1)^{n+k}a_k\,.
    \end{align*}
    Hence,
    \begin{align*}
        \sum_{x\in m^{-1}\Z^n\cap \R_{>0}^n}{h(x)}=\sum_{k=0}^n{b_km^k}+o(1)=\sum_{k=0}^n{(-1)^{n+k}a_km^k}+o(1)=(-1)^n\sum_{k=0}^n{a_k(-m)^k}+o(1)\,.
    \end{align*}
\end{proof}
\begin{remark}
    A related formula was obtained in Agapito and Weitsman \cite[below Theorem 3]{agapito_weighted_2005}.
    Equation \eqref{stanley_reciprocity} is a special case of Stanley reciprocity for rational cones (see Beck and Robins \cite[Theorem 4.3]{sinai_robins}).
\end{remark}

\begin{lemma}
    \label{quasi_vanishes}
    Let $p \colon \Z \to \R$
    be a quasi-polynomial function such that $p(m)=o(1)$ as $m \to \infty$. Then $p\equiv0$.
\end{lemma}
\begin{proof}
By definition, there exist $n_0\in \N$ and periodic functions $\{a_k \colon \Z \rightarrow \R\}_{k=0}^{n_0}$ such that for all $m \in \Z$
\[ p(m)=\sum_{k=0}^{n_0}{a_k(m)m^k}\,. \]
    Let $m^*=\prod_{k=0}^{n_0}{m_k}$ where $m_k\in \N$ is a period of $a_k$. Then for every $0 \le i < m^*$ the function $g_i(n)=p(i+nm^*)$ is polynomial in $n$.
    The assumption $p(m)=o(1)$ implies $g_i(n)=o(1)$, hence $g_i\equiv 0$ for every $0\le i< m^*$. Thus $p\equiv0$.
\end{proof}

\section{Integral-integral affine geometry}

\begin{definition}
    Let $M^n$ be a smooth manifold with corners. An \textbf{integral-integral affine atlas with corners} on $M$ is a collection of smooth charts with corners $\{\psi_i:U_i\rightarrow W_i\subseteq \R_{\ge 0}^n\}$ such that $\psi_j\circ\psi_i^{-1}$, when defined, locally agrees with an integral-integral affine map.
    An \textbf{integral-integral affine manifold with corners} is a smooth manifold with corners with a maximal (with respect to inclusion of integral-integral affine atlases with corners) integral-integral affine atlas with corners.
    A chart with corners from the maximal integral-integral affine atlas with corners is called an \textbf{integral chart}.
\end{definition}

\begin{remark}
    \label{measure_atlas}
    Let $M$ be an $n$-dimensional integral-integral affine manifold with corners. Since locally integral-integral affine maps preserve the standard Lebesgue measure on $\R^n$, there exists a unique {(Borel) measure} on $M$ such that each integral chart is measure preserving. The (affine) volume $\vol(M)$ is the measure of $M$ and integration on $M$ is done with respect to this measure.
\end{remark}

\begin{definition}
    Let $M$ be an $n$-dimensional integral-integral affine manifold with corners and let $k\in \N$. \textbf{The $k$-dilated manifold} $kM$ is M with a different maximal integral-integral affine atlas with corners. Let $\{\psi_i:U_i\rightarrow \R_{\ge 0}^n\}_{i\in I}$ be the maximal integral-integral affine atlas with corners of $M$ and let \begin{align*}
        &\phi_i:U_i\rightarrow \R_{\ge 0}^n\\
        &\phi_i(x)=k\cdot \psi_i(x)\,.
    \end{align*}
    The set $\{\phi_i\}_{i\in I}$ is a well-defined integral-integral affine atlas with corners since locally \begin{align*}
        &(\psi_i\circ \psi_j^{-1})(y)=w+Ay,\; w\in \Z^n,\; A\in GL_n(\Z)\\
        &(\phi_i\circ\phi_j^{-1})(y)=k\cdot\psi_i(\phi_j^{-1}(y))=k\cdot \psi_i(\psi_j^{-1}(\frac{1}{k}y))\\
        &=k(w+\frac{1}{k}Ay)=kw+Ay
    \end{align*}
    which is in $\AZGL{n}$ when $k\in \Z$.
    This atlas is contained in a unique maximal integral-integral affine atlas with corners $\mathcal{A}$. Then $kM\coloneq (M,\mathcal{A})$ is an integral-integral affine manifold with corners.
\end{definition}

\begin{proposition}
    \label{if_one_then_all}
      Let $M$ be an $n$-dimensional integral-integral affine manifold with corners Then
    \begin{multline*}
        M_{m^{-1}\Z} = \{x\in M \mid \text{ for every integral chart } \phi \colon U\rightarrow \Omega \subseteq \R^n_{\ge 0} \text{ whose domain contains $x$, } \\ \text{ we have } \phi(x)\in m^{-1}\Z^n\}.
    \end{multline*}
\end{proposition}
\begin{proof}
        Let $x \in M_{m^{-1}\Z}$.
So there exists an integral chart $\psi \colon U'\rightarrow \Omega' \subseteq \R^n_{\ge 0}$ 
whose domain contains $x$ and such that $\psi(x)\in m^{-1}\Z^n$.
Let $\phi \colon U \to \Omega \subseteq \R^n_{\ge 0}$
be another integral chart whose domain contains~$x$.
By definition
there exists a neighborhood of $x$
on which the composition
$\phi \circ \psi^{-1} \in \AZGL{n}$. Because $m^{-1}\Z^n$ is invariant under $\AZGL{n}$ we have $\phi(x)=\phi(\psi^{-1}(\psi(x)))=(\phi\circ\psi^{-1})(\psi(x)) \in m^{-1}\Z^n$.
\end{proof}

\begin{definition}
    Given an integral-integral affine manifold with corners $M$, the lattice-counting function of a subset  $A\subseteq M$ is \[
        L_{M,A}(m)\coloneq \#{(M_{m^{-1}\Z}\cap A)}\,.
    \]
    We omit $M$ when it is clear from context or when $A=M$.
\end{definition}

\begin{proposition}
    \label{integers_in_dilated_manifold}
    Let $M$ be an integral-integral affine manifold with corners. Then for every $k\in \N$
    \[
        M_{(km)^{-1}\Z}=(kM)_{m^{-1}\Z}\,.
    \]
\end{proposition}
\begin{proof}
    This follows from the definitions.
\end{proof}

\begin{corollary}
    \label{lattice_count_dilation}
    Let $M$ be an integral-integral affine manifold with corners. Then for every $k\in \N$ 
    \[
        L_{kM}(m)=L_M(km)\,.
    \]
\end{corollary}

\begin{definition}
\label{affine_image_polytope}
Let $P \subseteq M$ be a subset
of an $n$-dimensional integral-integral affine manifold with corners $M$. 
Then $P$ is called a \textbf{rational polytope} if there exists an integral chart $\phi$ whose domain contains $P$ and such that 
$\phi(P)$ is a rational polytope in $\R^n$.
\end{definition}
\begin{proposition}
\label{polytope_affine_chart}
Let M be an $n$-dimensional integral-integral affine manifold with corners, $P\subseteq M$ a rational polytope and $\psi \colon W \rightarrow \Omega \subseteq \R_{\ge 0}^n$ an integral chart such that $P\subseteq W$. Then $\psi(P) \subseteq \R_{\ge 0}^n$ is a rational polytope and $\dim \psi(P)$ is independent of $\psi$. 
\end{proposition}
\begin{proof}
    By definition, there is an integral chart $\phi \colon U \rightarrow \Omega'\subseteq  \R_{\ge 0}^n$ such that $P \subseteq U$ and $\phi(P)$ is a rational polytope.
In particular, $\phi(P)$ is connected.
Let $\calO$ be the connected component of $\phi(U \cap W)$
that contains $\phi(P)$.
On $\calO$ the map $\psi \circ \phi^{-1}$ is the restriction of some $f\in \AZGL{n}$. 
By Elisha, Karshon and Loizides \cite[Proposition 2.15]
{elisha2026integralpointsvolumeintegralintegral}
$\psi(P) = f(\phi(P))$ is a rational polytope. Since $f$ is an invertible affine map, dimensions of polytopes are preserved. That is, $\dim \psi(P)=\dim \phi(P)$.
\end{proof}

\begin{definition}
    Let $M$ be an integral-integral affine manifold with corners and $P\subseteq M$ a rational polytope. Its boundary is
    \[
        \partial P \coloneq \partial_M P \cup (P\cap \partial M)
    \]
    and its interior is \[
        \Int{P}\coloneq \Int_MP\setminus \partial M
    \]
    where $\partial_MP$ and $\Int_MP$ are the topological boundary and topological interior of $P$ respectively in $M$. 
\end{definition}

\begin{proposition}
    \label{integral_chart_operator_commute}
    Let $M$ be an integral-integral affine manifold with corners and let $P\subseteq M$ be a rational polytope. Then for every integral chart $\psi:U\rightarrow \R_{\ge 0}^n$ defined on a neighborhood of $P$ one has 
   \begin{align*}
       &\Int(\psi (P)) = \psi(\Int P) \\
        &\partial(\psi(P))=\psi(\partial P) \,,
   \end{align*}
   where the topological interior and boundary on the left are with respect to $\R^n$.
\end{proposition}
\begin{proof}
        Since $\psi(P)\subseteq \R^n_{\ge 0}$ we have
    \begin{align*}
        \psi(\Int P)&=\psi(\Int_MP\setminus \partial M)\\&=\psi((\Int_M P) \cap(\Int M\cap U))
        \\&=\Int_{\R_{\ge 0}^n}\psi(P)\cap \Int_{\R_{\ge 0}^n}{(\R_{>0}^n \cap \Omega)} \\
        &=\Int_{\R_{\ge 0}^n}{(\psi(P)\cap \R_{> 0}^n)}=\Int\psi(P)\,.
    \end{align*}
    Consequently,
    \begin{align*}
        \psi(\partial P)=\psi(P\setminus \Int P)=\psi(P)\setminus \psi(\Int P)=\psi(P)\setminus\Int\psi(P)=\partial (\psi(P))\,.
    \end{align*}
\end{proof}

\begin{proposition}
    \label{affine_polytope_reciprocity}
    Let $M$ be $n$-dimensional integral-integral affine manifold with corners and $P\subseteq M$ be an $n$-dimensional rational polytope. Then 
    \[
        L_{\Int{P}}(m)=(-1)^nL_P(-m)\,.
    \]
\end{proposition}
\begin{proof}
    Let $\psi:U\rightarrow \R_{\ge 0}^n$ be an integral chart where $P\subseteq U$. Then $\psi(P)$ is an $n$-dimensional rational polytope hence $\RInt \psi(P)=\Int \psi(P)$. By
    \Cref{rational_ehrhart},
    \begin{align}
        L_{\Int \psi(P)}(m)=L_{\RInt\psi(P)}(m)=(-1)^nL_{\psi(P)}(-m)\,.
    \end{align}
    Thus by \Cref{integral_chart_operator_commute}, and \Cref{if_one_then_all},
    \begin{align*}
        L_{\Int P}(m)=L_{\Int\psi(P)}(m)=(-1)^nL_{\psi(P)}(-m)=(-1)^nL_{P}(-m)\,.
    \end{align*}
\end{proof}

\begin{lemma}
    \label{intersection_lemma}
    Let $M$ be an $n$-dimensional integral-integral affine manifold with corners and $P_1,P_2 \subseteq M$ rational polytopes, and let $\phi_1:U_1\rightarrow \Omega_1\subseteq \R_{\ge 0}^n$ and $\phi_2:U_2 \rightarrow \Omega_2 \subseteq \R_{\ge 0}^n$ be integral charts such that $P_1\subseteq U_1$ and $P_2\subseteq U_2$. If $U_1\cap U_2$ is connected, then $P_1 \cap P_2$ is a rational polytope or is empty.
\end{lemma}
\begin{proof}
    By \Cref{polytope_affine_chart} the sets $\phi_1(P_1), \phi_2(P_2)$ are rational polytopes. 
    The intersection $U_1\cap U_2$ is connected so the composition $\phi_1\circ\phi_2^{-1}$ has the form $x \mapsto Ax +b$ where $b\in \Z^n$ and $A\in \mathsf{GL}_n(\Z)$. Then
    \begin{align*}
        &\phi_1(P_1 \cap P_2) = \phi_1(P_1 \cap P_2 \cap U_1 \cap U_2) = \phi_1(P_1)\cap\phi_1(P_2\cap U_1 \cap U_2).
    \end{align*}
    Moreover,
    \begin{multline*}
        \phi_1(P_2\cap U_1 \cap U_2)=(\phi_1\circ\phi_2^{-1})(\phi_2(P_2\cap U_1 \cap U_2))=b+A(\phi_2(P_2)\cap \phi_2(U_1\cap U_2)) \\
        =b+\left(A(\phi_2(P_2))\cap A(\phi_2(U_1\cap U_2))\right) =(b+A\phi_2(P_2))\cap(b+A\phi_2(U_1\cap U_2)).
    \end{multline*}
    Therefore,
    \begin{align*}
         \phi_1(P_1 \cap P_2)&=\phi_1(P_1)\cap (b+A\phi_2(P_2))\cap (b+A\phi_2(U_1\cap U_2))\\
         &=Q\cap (b+A\phi_2(U_1\cap U_2))
    \end{align*}
    where \[
        Q\coloneq\phi_1(P_1)\cap (b+A\phi_2(P_2)) \,.
    \]
    Both $P_1,P_2$ are compact subsets of $M$, hence $\phi_1(P_1\cap P_2)$ is compact. The set 
    $b+A\phi_2(U_1\cap U_2)$ is open in $\Omega_1$ and $Q\subseteq \Omega_1$ is a rational polytope \cite[Propositions 2.15, 2.18]
{elisha2026integralpointsvolumeintegralintegral}. So $\phi_1(P_1\cap P_2)$ is both an open and a closed subset of $Q$. Because $Q$ is (convex, hence) connected, $\phi_1(P_1\cap P_2)\in \{\emptyset, Q\}$. 
\end{proof}

\begin{lemma}{}
    \label{finite_rational_intersection}
    Let M be an $n$-dimensional integral-integral affine manifold with corners, let $\{P_i\}_{i\in I}$ be a finite collection of rational polytopes in $M$,
    and let $\{\phi_i \colon U_i\rightarrow \Omega_i \subseteq \R_{\ge 0}^n\}_{i\in I}$ be integral charts such that $P_i\subseteq U_i$ for every $i \in I$. Suppose that, for every subset $I' \subseteq I$, the intersection $\bigcap_{i \in I'} U_i$ is (empty or) connected. Then $\bigcap_{i \in I}{P_i}$, (if nonempty,) is a rational polytope.
\end{lemma}
\begin{proof}
    This follows by repeated application of \Cref{intersection_lemma}.
(When $I = \{1,\ldots,l\}$, it is enough
to assume that $U_1 \cap \ldots \cap U_{l'}$ is connected
for all $l' \in \{1,\ldots,l\}$.)
\end{proof}

\begin{definition}
    \label{good_cover}
    Let $M$ be an $n$-dimensional smooth manifold with corners. An open cover $\mathcal{U}\coloneq\{U_i\}_{i\in I}$ of $M$ is called a \textbf{good cover} if every nonempty finite intersection $\bigcap_{j=0}^m{U_j}$ of sets $U_j$ from $\mathcal{U}$ is contractible.
\end{definition}

\begin{theorem}
\label{manifold_corners_good_cover_corollary}
    Let $M$ be an $n$-dimensional smooth manifold with corners and let $\mathcal{U}\coloneq \{U_i\}_{i \in I}$ be an open cover of $M$. Then there exists a good cover of $M$ subordinate to $\mathcal{U}$. 
\end{theorem}
\begin{proof}
    By A. Douady and L. Hérault \cite[Theorem 6.1]{borel_corners_1973} there exists a smooth manifold with boundary $\bar M$ and a homeomorphism $f: \bar M \rightarrow M$.
    Let $\bar{\mathcal{U}}\coloneq \{f^{-1}(U_i)\}_{i\in I}$.
    By the metric argument in the proof of Boavida and Weiss \cite[Proposition 9.1]{boavida_de_brito_manifold_2013} (case where $k=1$), choosing the geodesically convex neighborhoods sufficiently small to refine \(\bar{\mathcal U}\),  there exists a good cover $\{\bar {V_j}\}_{j\in J}$ subordinate to $\bar{\mathcal{U}}$ hence $\{f(\bar{V_{j}})\}_{j\in J}$ is a good cover subordinate to $\mathcal{U}$.
\end{proof}
\begin{remark}
    This theorem is in principle known to experts (for example it is used in Volpe \cite[Proposition 2.19]{volpe_verdier_2025}), however the author did not find a detailed proof.
\end{remark}

\begin{corollary}{}
    \label{rational_polytope_cover}
    Let $M$ be a compact $n$-dimensional integral-integral affine manifold with corners. Then there exists a finite collection $\{P_i\}_{i=1}^k$ of $n$-dimensional rational polytopes such that: 
    \begin{enumerate}
        \item $M=\bigcup_{i=1}^k{P_i}$
        \item $\Int{M}=\bigcup_{i=1}^k{\Int{P_i}}$
        \item For every $\emptyset\neq I\subseteq \{1,\dots,k\}$, the intersection $\cap_{i\in I}{P_i}$ is a rational polytope or empty.
    \end{enumerate}
\end{corollary}
\begin{proof}
By \Cref{manifold_corners_good_cover_corollary}, there exists a good open cover $\calW$ of $M$
such that every element of $\calW$ is contained in the domain
of an integral chart and every finite intersection of elements of $\mathcal{W}$ is connected.

For each point $x \in M$, let $W_x$ be an element of $\calW$ that contains $x$, 
let $\varphi \colon U \to \Omega\subseteq \R_{\ge 0}^n$ be an integral chart 
whose domain contains $W_x$, 
and let $\tQ_x \subseteq \R^n$ be a closed cube with rational vertices
that is contained in $\varphi(W_x)$
and whose interior relative to the standard orthant contains $\varphi(x)$.
Then $Q_x\coloneq \varphi^{-1}(\tilde Q_x)$ is a rational polytope in $M$ and $x\in \Int_M(Q_x)$.

Because $M$ is compact, there exist $x_1,\ldots, x_k \in M$
such that 
$M=\bigcup_{i=1}^k{\Int_M(Q_{x_i})}$.

For each $i \in \{ 1, \ldots, k \}$, let $P_i\coloneq Q_{x_i}$. 
Then $\{P_i\}_{i=1}^k$ covers $M$ and 
\[
\bigcup_{i=1}^k{\Int{P_i}}=\bigcup_{i=1}^k{(\Int_M(P_i)\setminus \partial M)}=\left(\bigcup_{i=1}^k{\Int_M(P_i)}\right)\setminus\partial M=M\setminus \partial M=\Int M\,.
\]
For every $I \subseteq \{1,\ldots,k\}$, because the intersection 
$\bigcap_{i\in I} W_{x_i}$ is connected, and by \Cref{finite_rational_intersection}, the intersection $\bigcap_{i\in I}{P_i}$ is either empty or again a rational polytope.
\end{proof}

\section{Main Results}

\begin{lemma}
    \label{manifold_quasi_ehrhart}
    Let $M$ be a compact $n$-dimensional integral-integral affine manifold with corners. Then $L_M$, $L_{\Int M},L_{\partial M}$ are quasi-polynomials.
\end{lemma}
\begin{proof}
       Let $\{P_i\}_{i=1}^{N}$ be $n$-dimensional rational polytopes that cover $M$ such that $\bigcup_{i=1}^N\Int{P_i}=\Int M$ and every intersection is either empty or a rational polytope (\Cref{rational_polytope_cover}).
For each set of indices $\emptyset \neq I\subseteq \{1,\dots,N\}$,
denote $P_I\coloneq\bigcap_{i\in I}P_i$.
    By inclusion-exclusion,
    \begin{align*}
        L_M(m)=\sum_{x\in M_{m^{-1}\Z}}{1}=\sum_{
        \emptyset \neq I\subseteq\{1,\dots,N\}
         }{(-1)^{|I|+1}\# M_{m^{-1}\Z}\cap P_I}\,.
    \end{align*}
    For every $I\subseteq \{1,\ldots,N\}$ such that $P_I$ is non-empty, take an integral chart $\psi:U\rightarrow \Omega\subseteq \R^n$ such that $P_I\subseteq U$. By \Cref{polytope_affine_chart}, $\psi(P_I)$ is a rational polytope. By \Cref{if_one_then_all},
    \begin{align*}
        \#M_{m^{-1}\Z}\cap P_I=\#m^{-1}\Z^n\cap\psi(P_I).
    \end{align*}
By \Cref{rational_ehrhart} the RHS is quasi-polynomial in $m$, so 
$L_M(m)$ is an alternating sum of quasi-polynomials, hence a quasi-polynomial.

    By \Cref{affine_polytope_reciprocity}, if $\dim P_I=n$, then
    \[
        L_{P_I}(m)-L_{\partial P_I}(m)=L_{\Int{P_I}}(m)=(-1)^{n}L_{P_I}(-m)\,,
    \]
    hence $L_{\Int{P_I}},L_{\partial P_I}$ are quasi-polynomial.
    Since $I$ is finite,
    \begin{align*}
        \bigcap_{i\in I}{\Int P_i}=\bigcap_{i\in I}{((\Int_MP_i)\setminus \partial M)}=\left(\bigcap_{i\in I}{\Int_MP_i}\right)\setminus \partial M=\Int\left(\bigcap_{i\in I}P_i\right)\setminus \partial M=\Int P_I\,.
    \end{align*}
    By inclusion-exclusion on $\Int M$ and since $\Int{P_I}=\emptyset$ whenever $\dim P_I<n$,
    \begin{align}
        \label{full_dim_inclusion_exclusion}
        &L_{\Int{M}}(m)=\sum_{\emptyset \neq I\subseteq [N]}{(-1)^{|I|+1}L_{\Int{P_I}}(m)}=\sum_{\substack{\emptyset \neq I\subseteq [N]\\dim{P_I}=n}}{(-1)^{|I|+1}L_{\Int{P_I}}(m)}
    \end{align}
    so $L_{\Int M},L_{\partial M}$ are sums of quasi-polynomials, hence quasi-polynomial.
\end{proof}

\begin{theorem}
    \label{generalized_ehrhart}
    Let $M$ be a compact $n$-dimensional integral-integral affine manifold with corners. Then its lattice-counting function $L_M$ is polynomial and \[
        L_{\Int M}(m)=(-1)^{n}L_M(-m)\,.
    \]
\end{theorem}
\begin{proof}
    Because $M$ is compact, it has a finite integral-integral affine atlas with corners.
Let $\{\rho_i\}_{i=1}^{N}$ be a partition of unity
that is subordinate to the cover of $M$ by the domains of the charts
in such an atlas.
For each $i \in \{1,\ldots,N\}$,
let $\psi_{i} \colon U_{i} \to \Omega_{i} \subseteq \R^n_{\ge 0}$
be an integral chart whose domain contains
the support $\supp \rho_i$ of $\rho_i$ in $M$.
\onote{Do we really need this? I think choosing an extendable partition of unity is more general than extending with Whitney}
Then the subset $\psi_{i}(\supp \rho_i)$ of $\Omega_i$ is compact,
so $\rho_i \circ \psi_{i}^{-1}$
extends by zero to a compactly supported
smooth function on $\R^n_{\ge 0}$ which can be extended by the Whitney extension theorem to a smooth compactly supported function $h_i$ on $\R^n$.
By \Cref{orthant_asymptotics} there exist $a_{i,0}\dots a_{i,n}\in \R$ such that
\begin{align*}
        &
        \sum_{x\in M_{m^{-1}\Z}}{\rho_i(x)}=\sum_{y\in m^{-1}\Z^n \cap \Omega_i }{\rho_i(\psi^{-1}_{i}(y))}=\sum_{y\in m^{-1}\Z^n\cap \R_{\ge 0}^n}{h_i(y)}=\sum_{k=0}^n{a_{i,k}m^k} +o(1)
\end{align*}
as $m\rightarrow \infty$.
    Thus
    \begin{align*}
        &L_M(m)=\#M_{m^{-1}\Z}=\sum_{x\in M_{m^{-1}\Z}}{1}=\sum_{x\in M_{m^{-1}\Z}}{\sum_{i=1}^{N}{\rho_i(x)}}=\sum_{i=1}^{N}{\sum_{x\in M_{m^{-1}\Z}}{\rho_i(x)}}\\
        &=\sum_{i=1}^{N}{\left(\sum_{k=0}^n{a_{i,k}m^k\, +o(1)}\right)}
        =\sum_{k=0}^n{\left(\sum_{i=1}^N{a_{i,k}}\right)m^k}\,+o(1)\,.
    \end{align*}
    $L_M$ is quasi-polynomial by \Cref{manifold_quasi_ehrhart} so by \Cref{quasi_vanishes}, \[
  L_M(m)=\sum_{k=0}^n{\left(\sum_{i=1}^N{a_{i,k}}\right)m^k}\,.
    \]

    By \Cref{local_reciprocity},
      \begin{align*}
        &L_{\Int M}(m)=\sum_{x\in (\Int M)_{m^{-1}\Z}}{1}
        =\sum_{x\in(\Int M)_{m^{-1}\Z}}\sum_{i=1}^{N}{\rho_i(x)}=\sum_{i=1}^{N}{\sum_{x\in(\Int M)_{m^{-1}\Z}}{\rho_i(x)}}\\
        &=\sum_{i=1}^{N}{\sum_{y\in m^{-1}\Z^n\cap\R_{>0}^n\cap\Omega_i}{\rho_i(\psi_i^{-1}(y))}}=\sum_{i=1}^{N}{\sum_{y\in m^{-1}\Z^n\cap \R_{>0}^n}{h_i(y)}}
        \\&=\sum_{i=1}^{N}{(-1)^n\sum_{k=0}^na_{i,k}(-m)^k+o(1)}\\&
        =(-1)^n\sum_{k=0}^n{\left(\sum_{i=1}^Na_{i,k}\right)(-m)^k}+o(1)=(-1)^nL_M(-m)+o(1)\,.
    \end{align*}
    $L_{\Int M}$ is quasi-polynomial by \Cref{manifold_quasi_ehrhart} 
    so by \Cref{quasi_vanishes},
    \[
        L_{\Int{M}}(m)=(-1)^nL_{M}(-m)\,.
    \]
\end{proof}

\begin{corollary}
    Let $M$ be a compact $n$-dimensional integral-integral affine manifold with corners. Then $L_{\partial M}$ is a polynomial with \[
        L_{\partial M}(m)=L_{M}(m)+(-1)^{n+1}L_M(-m)\,.
    \]
\end{corollary}
\begin{proof}
    By \Cref{generalized_ehrhart}
    \begin{align*}
        &L_{\partial M}(m)=L_{M}(m)-L_{\Int M}(m)=L_M(m)-(-1)^nL_M(-m)\\
        &=L_{M}(m)+(-1)^{n+1}L_{M}(-m)\,.
    \end{align*}
\end{proof}

\begin{corollary}
    \label{leading_constant_coefficients}
    Let $M$ be a compact $n$-dimensional integral-integral affine manifold with corners and let $L_M(m)=\sum_{k=0}^n{a_km^k}$ be its lattice-counting function. Then
    \begin{enumerate}
        \item $a_n=\Vol(M)$ (see \Cref{measure_atlas}) \label{coefficients:volume}
        \item $a_0=\chi(M)$ \label{coefficients:euler}
    \end{enumerate}
\end{corollary}
\begin{proof}
    (\ref{coefficients:volume}) Let $\{P_i\}_{i=1}^N$ be rational polytopes that cover $M$ (\Cref{rational_polytope_cover}). For every $\emptyset\neq I\subseteq [N]$, set $P_I\coloneq \bigcap_{i\in I}{P_i}$, and let $L_{P_I}(m)=\sum_{k=0}^n{a_k^{P_I}(m)m^k}$ be its lattice-counting function (\Cref{rational_ehrhart}). By inclusion-exclusion, \begin{align*}
        &L_M(m)=\sum_{\emptyset \neq I\subseteq [N]}{(-1)^{|I|+1}L_{P_I}(m)}\,.
    \end{align*}
    Hence, 
    \begin{align*}
        a_n=\sum_{\emptyset \neq I\subseteq [N]}{(-1)^{|I|+1}a^{P_I}_n(m)}\,.
    \end{align*}
    By \Cref{rational_ehrhart} the $n$-th coefficient of the lattice-counting function of every rational polytope in $\R^n$ is its $n$-th dimensional volume (even when the polytope is not full dimensional). Thus,
    \begin{align*}
        a_n=\sum_{\emptyset\neq I\subseteq [N]}{(-1)^{|I|+1}\Vol(P_I)}=\Vol(\bigcup_{i=1}^N{P_i})=\Vol(M)\,.
    \end{align*}

    (\ref{coefficients:euler}) Let $k_0\in \N$ such that for every $\emptyset \neq I\subseteq [N]$ the vertices of $P_I$ lie on $M_{k_0^{-1}\Z}$. By \Cref{integers_in_dilated_manifold} each $P_I$ has integer vertices under any integral chart of $k_0M$.
    By \Cref{rational_ehrhart}, whenever $P_I$ is non-empty, the constant coefficient of $L_{P_I}$ with respect to $k_0M$ is $1$, which coincides with the Euler characteristic of $P_I$. If $P_I$ is empty then $\chi(P_I)=0$.
    By inclusion-exclusion the constant coefficient of $L_{k_0M}$ is 
    \begin{align*}
        \sum_{\substack{\emptyset\neq I\subseteq [N]\\ \emptyset \neq P_I}}(-1)^{|I|+1}=\sum_{\emptyset \neq I\subseteq [N]}{(-1)^{|I|+1}\chi(P_I)}=\chi(\bigcup_{i=1}^N{P_i})=\chi(M)
    \end{align*}
     By \Cref{lattice_count_dilation} the polynomial $L_{kM}$ has the same constant coefficient for every $k\in \N$ hence $L_M$ has constant coefficient $\chi(M)$.
\end{proof}
\begin{corollary}
    Let M be a closed $n$-dimensional integral-integral affine manifold with $n\ge 1$. Then $\chi(M)=0$.
\end{corollary}
\begin{proof}
    By \Cref{leading_constant_coefficients} the constant coefficient of $L_M$ is $\chi(M)$ and by Elisha, Karshon and Loizides \cite[Corollary 3.14]{elisha2026integralpointsvolumeintegralintegral} the constant coefficient of $L_M$ is zero whenever $n\ge 1$.
\end{proof}
\begin{remark}
    This is a special case of Chern's conjecture. See, e.g., Klingler \cite{klingler_cherns_2017} or Arias Abad and Vélez Vásquez \cite{abad_lectures_2025}.
\end{remark}

\appendix
\printbibliography
\end{document}